\documentclass[letterpaper, 10 pt, conference]{ieeeconf}  

\IEEEoverridecommandlockouts                              
\usepackage{amssymb}
\usepackage{amsmath}
\usepackage{xcolor}
\usepackage{graphicx}
\usepackage{cuted}

\makeatletter

\newtheorem{thm}{\protect\theoremname}
\newtheorem{defn}[thm]{\protect\definitionname}
\newtheorem{lem}[thm]{\protect\lemmaname}
\newtheorem{cor}[thm]{\protect\corollaryname}
\newtheorem{example}[thm]{\protect\examplename}
\newtheorem{prop}[thm]{\protect\propname}
\newtheorem{rem}[thm]{\protect\remarkname}

\makeatother

\providecommand{\corollaryname}{Corollary}
\providecommand{\definitionname}{Definition}
\providecommand{\examplename}{Example}
\providecommand{\lemmaname}{Lemma}
\providecommand{\theoremname}{Theorem}
\providecommand{\propname}{Proposition}
\providecommand{\remarkname}{Remark}

\renewcommand\labelenumi{(\roman{enumi})}
\renewcommand\theenumi\labelenumi

\DeclareMathOperator{\Rank}{rank}
\DeclareMathOperator{\Span}{span}

\title{\LARGE \bf
Flatness of Two-Input Discrete-Time Systems and their Linearization
}

\author{Johannes Schrotshamer, Bernd Kolar and Markus Sch{\"o}berl
\thanks{This research was funded in whole, or in part, by the Austrian Science Fund (FWF) P36473. For the purpose of open access, the author has applied a CC BY public copyright license to any Author Accepted Manuscript version arising from this submission.}
\thanks{All authors are with the Institute of Control Systems, Johannes Kepler University Linz, Altenberger Strasse 69, 4040 Linz, Austria (e-mail: johannes.schrotshamer@jku.at, bernd.kolar@jku.at, markus.schoeberl@jku.at).}%
}

\begin{document}

\maketitle
\thispagestyle{empty}
\pagestyle{empty}

\begin{abstract}
In this contribution we discuss flat discrete-time nonlinear systems in a general setting including two special subclasses, namely, forward- and backward-flat systems. We relate rank conditions for certain submatrices of the Jacobian of the flat parameterization to the mentioned subclasses. Motivated by these rank conditions, for the case of two-input systems that possess an $(x,u)$-flat output, we derive a simple type of dynamic extension for the purpose of an exact linearization.
\end{abstract}


\section{INTRODUCTION}

    In the 1990s, Fliess, L{\'e}vine, Martin and Rouchon introduced the concept of flatness for nonlinear continuous-time systems (see e.g. \cite{FliessLevineMartinRouchon:1995} and \cite{FliessLevineMartinRouchon:1999}). Flat continuous-time systems possess the characteristic property that all system variables can be expressed by a flat output and its time derivatives. Due to its high practical relevance, flatness is doubtlessly one of the most important concepts in nonlinear control theory.
    Although certain subclasses of differentially flat systems can already be systematically checked for flatness (see e.g. \cite{NicolauRespondek:2017}, \cite{GstottnerKolarSchoberl:2021b}, \cite{Levine:2025}, or \cite{NicolauRespondekLi:2025}), a computationally efficient test in the form of verifiable necessary and sufficient conditions for nonlinear multi-input systems does currently not exist.
    
    For the practical application it has been shown e.g. in \cite{DiwoldKolarSchoberl:2022} that it can be advantageous to design a discrete-time flatness-based controller for a suitable discretization of the continuous-time system. Regarding the flatness of nonlinear discrete-time systems, in the literature there exist two approaches.
    The first one adopted e.g. in \cite{Sira-RamirezAgrawal:2004}, \cite{KaldmaeKotta:2013}, or \cite{KolarKaldmaeSchoberlKottaSchlacher:2016} is to simply replace the time derivatives of the continuous-time definition by forward-shifts. For these so-called forward-flat systems, a geometric test in the form of necessary and sufficient conditions has been derived in \cite{KolarDiwoldSchoberl:2023}.
    The second approach proposed in \cite{DiwoldKolarSchoberl:2020} considers flatness as the one-to-one correspondence of the system trajectories and trajectories of a trivial system. This method represents a generalization of the first approach since the flat output may not only depend on the system variables and their forward-shifts, but also on their backward-shifts (see also \cite{GuillotMillerioux2020} or \cite{Kaldmae:2021}). Although, in the recent paper \cite{Kaldmae:2021} necessary and sufficient conditions for the flatness of discrete-time systems using an algebraic framework are proposed, due to the occurrence of degrees of freedom and the need to solve partial differential equations, which leads to a high computational complexity, alternative methods are still being worked on.\\
    In the continuous-time case, it is well-known that every flat system can be exactly linearized by applying a suitable dynamic feedback, see e.g. \cite{FliessLevineMartinRouchon:1999}.
    Additionally, subclasses of flat systems that can be linearized by simple types of dynamic feedback like prolongations of the original or transformed input have been studied.
    Knowing that such an exact linearization by prolongations is possible has proven extremely helpful in deriving computationally efficient tests for the flatness of the corresponding system classes as e.g. in \cite{NicolauRespondek:2017}, \cite{GstottnerKolarSchoberl:2021b}, \cite{Levine:2025}, or \cite{GstoettnerKolarSchoeberl:2021}.\\
    In \cite{DiwoldKolarSchoberl:2020} it has been shown that also for flat discrete-time systems an exact linearization by a suitable dynamic feedback is always possible, see also \cite{KolarDiwoldGstottnerSchoeberl:2022} for a more detailed discussion with a focus on tracking control design. Further results regarding dynamic feedback linearization of discrete-time systems in general can be found e.g. in \cite{Aranda-BricaireMoog:2008} or \cite{KaldmaeKotta:2025}.
    As already mentioned above, for discrete-time systems so far only the property of forward-flatness can be checked in a computationally efficient manner, whereas the general case remains challenging. Inspired by the above mentioned results from the continuous-time case, as a first step towards a computationally feasible test for discrete-time flatness, in this contribution we show that flat discrete-time systems with two inputs and a certain type of flat output can be exactly linearized via a combination of prolongations and a similarly simple type of dynamic feedback which we shall refer to as ''prelongations''.
    For this purpose, we introduce the subclass of backward-flat systems, which --- analogously to the special case of forward-flat systems --- arises naturally from the definition of flatness when only backward-shifts occur both in the flat output and the flat parameterization of the system variables.
    Subsequently, we show that the properties of forward- and backward-flatness are related to rank conditions regarding certain submatrices of the Jacobian of the flat parameterization.
    Motivated by these rank conditions, for systems with two inputs and a certain type of flat output, we first prove that in the forward-flat case an exact linearization by prolongations of a suitably chosen transformed input is possible. This can be shown similarly as in \cite{GstottnerKolarSchoberl:2020} for the continuous-time setting. Second, we prove that in the backward-flat case an exact linearization by the above mentioned prelongations is possible. Finally, we show that systems that are neither forward- nor backward-flat still allow an exact linearization by a combination of these special types of dynamic feedback.

    The paper is organized as follows: In Section \ref{sec:flatness}, we first recall the concept of flatness for discrete-time systems as discussed in \cite{DiwoldKolarSchoberl:2020}. In addition to the subclass of forward-flat systems, we also introduce the subclass of backward-flat systems. Next, in Section \ref{sec:properties_Jacobian}, we relate rank conditions for certain submatrices of the Jacobian of the flat parameterization to the mentioned subclasses. Finally, motivated by these rank conditions, in Section \ref{sec:linearization}, for the case of two-input systems that possess an $(x,u)$-flat output we derive a simple type of dynamic extension for the purpose of an exact linearization and illustrate our results by an academic example and the well-known practical examples of a planar VTOL aircraft and a wheeled mobile robot.

    \paragraph*{Notation}
    To keep formulas short and readable, the one-fold forward-shift of a variable is sometimes denoted by a superscript +, while backward- and higher-order forward-shifts are generally represented using subscripts in brackets, e.g. $u_{[\alpha]}$ denotes the $\alpha$th forward-shift and $u_{[-\alpha]}$ denotes the $\alpha$th backward-shift of $u$ with $\alpha\geq0$.


\section{FLATNESS OF DISCRETE-TIME SYSTEMS}\label{sec:flatness}
    Throughout this contribution, we consider nonlinear time-invariant discrete-time systems of the form
	\begin{equation}\label{eq:sysEq}
		x^{i,+}=f^{i}(x,u)\,,\quad i=1,\dots,n
	\end{equation}
	with $\dim(x)=n$, $\dim(u)=m$ and smooth functions $f^{i}(x,u)$ with $\Rank (\partial_u f)=m$ (independent inputs). Additionally, we assume
    that the system equations \eqref{eq:sysEq} satisfy the submersivity condition $\Rank(\partial_{(x,u)}f)=n$, which is a common assumption in the discrete-time literature, as it is necessary for accessibility (see e.g. \cite{Grizzle:1993}).\\
    According to \cite{DiwoldKolarSchoberl:2020}, a discrete-time system \eqref{eq:sysEq} is flat if there exists a one-to-one correspondence between its solutions $(x(k),u(k))$ and solutions $y(k)$ of a trivial system (arbitrary trajectories that need not satisfy any difference equation) with the same number of inputs. The one-to-one correspondence of the solutions means that the values of $x(k)$ and $u(k)$ at some fixed time step $k$ may depend on an arbitrary but finite number of future and past values of $y(k)$. On the contrary, the value of $y(k)$ at some fixed time step $k$ may depend on an arbitrary but finite number of future and past values of $x(k)$ and $u(k)$.\\
    Because of the time-invariance of the system \eqref{eq:sysEq}, the one-to-one correspondence can be expressed by maps of the form
    \begin{equation*}\label{eq:one_to_one_xu}
        (x(k),u(k))=F(y(k-l_1),\ldots,y(k),\ldots,y(k+l_2))
    \end{equation*}
    and
    \begin{equation*}\label{eq:one_to_one_y}
        \begin{split}
            y(k)=&\varphi(x(k-p_1),u(k-p_1),\ldots,x(k),u(k),\ldots,\\
            &x(k+p_2),u(k+p_2))
        \end{split}
    \end{equation*}
    with suitable integers $l_1,l_2,p_1,p_2$.
    Through repeated application of \eqref{eq:sysEq}, all forward-shifts $x(k+\alpha), \, \alpha \geq 1$ are obviously determined by $x(k)$ and the input trajectory $u(k),\ldots,u(k+\alpha-1)$. A similar argument holds for the backward direction: Due to the submersivity condition $\Rank(\partial_{(x,u)}f)=n$, there always exist $m$ functions $g(x, u)$ such that the map
    \begin{equation}\label{eq:sysEq_ext}
		x^{+} = f(x, u)\,,\quad \zeta = g(x, u)
	\end{equation}
	is locally a diffeomorphism and hence invertible. By a repeated application of its inverse
    \begin{equation}\label{eq:sysEq_ext_inv}
		x=\psi_x(x^+,\zeta)\,, \quad u=\psi_u(x^+,\zeta)\,,
	\end{equation}
    all backward-shifts $x(k-\beta),\, u(k-\beta)$ of the state- and input variables for $\beta \geq 1$ are determined by $x(k)$ and backward-shifts $\zeta(k-1),\ldots,\zeta(k-\beta)$ of the system variables $\zeta$.\\
    Additionally, we make use of a space with coordinates $(\ldots,\zeta_{[-1]},x,u,u_{[1]},\ldots)$. Future values of a smooth function $h$ defined on this space can, in accordance with \eqref{eq:sysEq_ext}, be determined by a repeated application of the forward-shift operator $\delta$ defined by the rule
    \begin{equation*}
    \begin{split}
        & \delta(h(\ldots,\zeta_{[-2]},\zeta_{[-1]},x,u,u_{[1]},\ldots))=\\
        & \qquad h(\ldots,\zeta_{[-1]},g(x,u),f(x,u),u_{[1]},u_{[2]},\ldots)\,.
    \end{split}
    \end{equation*}
    Due to \eqref{eq:sysEq_ext_inv}, past values can in turn be determined by a repeated application of its inverse, the backward-shift operator $\delta^{-1}$ defined by\footnote{The forward- and backward-shift operators on a space with the coordinates $(\ldots,y_{[-1]},y,y_{[1]},\ldots)$ are defined accordingly.}
    \begin{equation*}
    \begin{split}
        & \delta^{-1}(h(\ldots,\zeta_{[-1]},x,u,u_{[1]},u_{[2]},\ldots))=\\
        & \qquad h(\ldots,\zeta_{[-2]},\psi_x(x,\zeta_{[-1]}),\psi_u(x,\zeta_{[-1]}),u,u_{[1]},\ldots)\,.
    \end{split}
    \end{equation*}
    With these preliminaries, we can give a concise definition of the concept of flatness for discrete-time systems, which is a modified but equivalent formulation of Definition 1 in \cite{DiwoldKolarSchoberl:2020}.
	\begin{defn}
		\label{def:flatness}The system \eqref{eq:sysEq} is said to be flat around an equilibrium $(x_{0},u_{0})$, if the $n+m$ coordinate functions $x$ and $u$ can be expressed locally by an $m$-tuple of functions
		\begin{equation}
			y^{j}=\varphi^{j}(\zeta_{[-Q_1,-1]},x,u,u_{[1,Q_2]})\,,\quad\,j=1,\ldots,m\label{eq:flat_output}
		\end{equation}
		and their backward- and forward-shifts up to some finite order\footnote{The multi-index $R_1=(r_{1,1},\ldots,r_{1,m})$ contains the number of backward-shifts and $R_2=(r_{2,1},\ldots,r_{2,m})$ contains the number of forward-shifts of the individual components of the flat output which are needed to express $x$ and $u$, and $y_{[-R_1,R_2]}$ is an abbreviation for $y$ and its backward-shifts up to order $R_1$ and its forward-shifts up to order $R_2$.}, i.e.
    	\begin{equation}\label{eq:flat_param}
    		\begin{split}
    			x&=F_{x}(y_{[-R_1,R_2-1]})\,,\\
    			\:u&=F_{u}(y_{[-R_1,R_2]})\,.
    		\end{split}
    	\end{equation}
        The $m$-tuple \eqref{eq:flat_output} is called a flat output and the parameterization \eqref{eq:flat_param} of $x$ and $u$ is unique.
	\end{defn}

    If we restrict ourselves to forward-shifts in the flat output \eqref{eq:flat_output} and the parameterization \eqref{eq:flat_param}, i.e. $Q_1=0$ and $R_1=0$, then Definition \ref{def:flatness} leads to the special case of forward-flatness.
	\begin{defn}
		\label{def:forward-flatness}The system \eqref{eq:sysEq} is said
		to be forward-flat around an equilibrium $(x_{0},u_{0})$, if the
		$n+m$ coordinate functions $x$ and $u$ can be expressed locally
		by an $m$-tuple of functions
		\begin{equation}
			y^{j}=\varphi^{j}(x,u,u_{[1,Q_2]})\,,\quad j=1,\ldots,m\label{eq:forward-flat_output}
		\end{equation}
		and their forward-shifts up to some finite order, i.e.
    	\begin{equation}\label{eq:flat_param_forward}
    		\begin{split}
    			x&=F_{x}(y_{[0,R_2-1]})\,,\\
    			\:u&=F_{u}(y_{[0,R_2]})\,.
    		\end{split}
    	\end{equation}
        The $m$-tuple \eqref{eq:forward-flat_output} is called a forward-flat output and the parameterization \eqref{eq:flat_param_forward} of $x$ and $u$ is unique.
	\end{defn}

    \begin{rem}
        As shown in \cite{KolarDiwoldSchoberl:2023}, forward-flatness can be systematically checked by computing a unique sequence of involutive distributions. This sequence generalizes the sequence of distributions from the static feedback linearization test in \cite{Grizzle:1986}, which is why the latter is included as a special case. It should be noted that there does not exist a counterpart in the continuous-time case.
    \end{rem}
	
	Conversely, if we only permit backward-shifts to appear in the flat output \eqref{eq:flat_output} and the parameterization \eqref{eq:flat_param}, i.e. $Q_2=0$ and $R_2=0$, then Definition \ref{def:flatness} leads to a subclass of flat discrete-time systems which we call backward-flat systems.
	\begin{defn}
		\label{def:backward-flatness}The system \eqref{eq:sysEq} is said
		to be backward-flat around an equilibrium $(x_{0},u_{0})$, if the
		$n+m$ coordinate functions $x$ and $u$ can be expressed locally
		by an $m$-tuple of functions
		\begin{equation}
			y^{j}=\varphi^{j}(\zeta_{[-Q_1,-1]},x,u)\,,\quad j=1,\ldots,m\label{eq:backward-flat_output}
		\end{equation}
		and their backward-shifts up to some finite order, i.e.
        \begin{equation}\label{eq:flat_param_back}
    		\begin{split}
    			x&=F_{x}(y_{[-R_1,-1]})\,,\\
    			\:u&=F_{u}(y_{[-R_1,0]})\,.
    		\end{split}
	    \end{equation}
        The $m$-tuple \eqref{eq:backward-flat_output} is called a backward-flat output and the parameterization \eqref{eq:flat_param_back} of $x$ and $u$ is unique.
	\end{defn}
    For flat systems according to Definition \ref{def:flatness}, Definition \ref{def:forward-flatness} and Definition \ref{def:backward-flatness}, with $R=R_1+R_2$ the inequality $\# R\geq n$ applies and the parameterizing map $\left(F_x, F_u\right): \mathbb{R}^{\# R+m} \rightarrow \mathbb{R}^{n+m}$ is a submersion.\footnote{The addition $R=R_1+R_2$ is considered element-wise as $R=(r_1,\ldots,r_m)=(r_{1,1}+r_{2,1},\ldots,r_{1,m}+r_{2,m})$ and $\# R=r_1+\ldots+r_m$.} We denote the difference of the dimensions of the domain and the co-domain of this map by $d$, i. e. $d=\# R+m-(n+m)=\# R-n$. The difference $d$ of a flat output $y$ is the minimal dimension of the state $z$ of a linearizing dynamic feedback, see Section \ref{sec:linearization}.
    \begin{rem}
        If $\# R=n$ holds, the parameterizing map $\left(F_x, F_u\right): \mathbb{R}^{n+m} \rightarrow \mathbb{R}^{n+m}$ is a diffeomorphism and the system can be exactly linearized by static feedback. We refer to these particular flat outputs as linearizing outputs.
    \end{rem}

    \section{The Jacobian Matrix of the Parameterizing Map} \label{sec:properties_Jacobian}
    In this section, we will show that for flat systems according to Definition \ref{def:flatness} with a flat output \eqref{eq:flat_output} (including the two special cases of forward- and backward-flatness), the submatrices $\partial_{y_{[-R_1]}}F_x$ and $\partial_{y_{[R_2]}}F_u$ of the Jacobian matrix
    \begin{equation}\label{eq:Jacobian1}
        \resizebox{.99 \hsize}{!}{
        $\displaystyle
            \partial_{y_{[-R_1, R_2]}} F_{x,u}=\left[\begin{array}{ccc}
    		\partial_{y_{[-R_1]}} F_x & \partial_{y_{[-R_1+1, R_2-1]}} F_x & \underline{0} \\
    		\partial_{y_{[-R_1]}} F_u & \partial_{y_{[-R_1+1, R_2-1]}} F_u & \partial_{y_{[R_2]}} F_u
    	\end{array}\right]
        $}
    \end{equation}
     of the parameterizing map \eqref{eq:flat_param} exhibit a distinct rank.\\
    Because of $\Rank (\partial_u f)=m$, it is always possible to apply an invertible input transformation as stated in the following lemma, which results in a special structure of the Jacobian matrix \eqref{eq:Jacobian1}.
    \begin{lem}\label{lem:Jacobian_transformed}
        Consider a system \eqref{eq:sysEq} that is flat according to Definition \ref{def:flatness} with a flat output \eqref{eq:flat_output}. After an invertible input transformation $v^j= f^j(x, u)$ such that $m$ of the $n$ system equations are normalized in the form $x^{j,+}=v^j$, the Jacobian of the corresponding parameterizing map $(x, v)=F_{x, v}\left(y_{\left[-R_1, R_2\right]}\right)$ has the form
        \begin{equation}\label{eq:Jacobian_transformed}
        \resizebox{.99 \hsize}{!}{
        $\displaystyle
            \partial_{y_{\left[-R_1, R_2\right]}} F_{x, v}=\left[\begin{array}{ccc}
            \partial_{y_{\left[-R_1\right]}} F_x & \partial_{y_{\left[-R_1+1, R_2-1\right]}} F_x & \underline{0} \\
            \underline{0} & \partial_{y_{\left[-R_1+1, R_2-1\right]}} F_v & \partial_{y_{\left[R_2\right]}} F_v
            \end{array}\right]\,.
        $}
        \end{equation}
    \end{lem}
    \begin{proof}
        Because of $x^{j,+}=v^j$, it is obvious that $F_v$ cannot depend on the highest backward-shifts $y_{[-R_1]}$ that occur in $F_x$. Thus, $\partial_{y_{[-R_1]}}F_v=\underline{0}$ holds.
    \end{proof}
    Since the flat output \eqref{eq:flat_output} of a flat system, as defined in Section \ref{sec:flatness}, is in general a function of the coordinates $(\zeta_{[-Q_1,-1]},x,u,u_{[1,Q_2]})$, with \eqref{eq:sysEq_ext} it is obvious that the variables $y_{[-R_1,R_2]}$ can be expressed as functions of $F_x$, $F_u$, backward-shifts $\delta^{-Q_1-R_1}g(F_x,F_u),\ldots,\delta^{-1}g(F_x,F_u)$ and forward-shifts $\delta F_u,\ldots,\delta^{Q_2+R_2} F_u$. This implies
    \begin{equation}\label{eq:differentials}
        \begin{split}
            \Span & \{\mathrm{d}y_{[-R_1,R_2]}\} \subset \Span \{\delta^{-M_1}\mathrm{d}g(F_x,F_u),\ldots,\\
            &\delta^{-1}\mathrm{d}g(F_x,F_u),\mathrm{d}F_x,\mathrm{d}F_u,\delta\mathrm{d}F_u,\ldots,\delta^{M_2}\mathrm{d}F_u\}
        \end{split}
    \end{equation}
    with some minimal $M_1\leq Q_1+R_1$ and some minimal $M_2\leq Q_2+R_2$.\\
    With these results, we can state the following proposition.
    \begin{prop}\label{prop:differentials}
        The minimal multi-indices $M_1$ and $M_2$ in \eqref{eq:differentials} are related to the submatrices $\partial_{y_{[R_2]}}F_u$ and $\partial_{y_{[-R_1]}}F_x$ as follows:
        \begin{enumerate}
            \item The submatrix $\partial_{y_{[R_2]}}F_u$ is of full rank, i.e. $\Rank (\partial_{y_{[R_2]}}F_u)=m$ if and only if $M_2=0$.
            \item The submatrix $\partial_{y_{[-R_1]}}F_x$ is of full rank, i.e. $\Rank (\partial_{y_{[-R_1]}}F_x)=m$ if and only if $M_1=0$.
        \end{enumerate}
    \end{prop}
    
    \begin{proof}
        \\(i) Direction $\Rank (\partial_{y_{[R_2]}}F_u)=m\,\Rightarrow\,M_2=0$: For arbitrary $\alpha\geq0$, the $\alpha$th forward-shift of the 1-forms $\mathrm{d}F_u$ has the form 
        \begin{equation}\label{eq:differential_forward}
    		\delta^{\alpha}\mathrm{d} F_u=\ldots+\delta^{\alpha} (\partial_{y_{[R_2+\alpha]}} F_u) \mathrm{d} y_{[R_2+\alpha]}\,,
    	\end{equation}
        where the coefficients of the differentials $\mathrm{d} y_{[R_2+\alpha]}$ of the highest occurring forward-shifts of the flat output are given by the elements of the $\alpha$th forward-shift of the matrix $\partial_{y_{[R_2]}} F_u$. Thus, if $\partial_{y_{[R_2]}} F_u$ is regular\footnote{In a sufficiently small neighborhood of an equilibrium, applying the shift-operator does not change the rank of the Jacobian matrix.}, then for every $\alpha \ge 0$ there does not exist a linear combination of the 1-forms $\delta^{\alpha} \mathrm{d} F_u$ where the differentials $\mathrm{d} y_{[R_2 + \alpha]}$ cancel out. Hence, the condition \eqref{eq:differentials} must hold also without 1-forms \eqref{eq:differential_forward} with $\alpha\geq1$, and $M_2=0$ follows.

        Direction $M_2=0\,\Rightarrow\,\Rank (\partial_{y_{[R_2]}}F_u)=m$: For $M_2=0$, the $m$ differentials $\mathrm{d}y_{[R_2]}$ of the left-hand side of \eqref{eq:differentials} only appear in the $m$ 1-forms $\mathrm{d}F_u$ of the right-hand side of \eqref{eq:differentials}, with their coefficients given by the elements of the Jacobian $\partial_{y_{[R_2]}}F_u$. Thus, \eqref{eq:differentials} can only hold if $\Rank (\partial_{y_{[R_2]}}F_u)=m$.

        (ii) Direction $\Rank (\partial_{y_{[-R_1]}}F_x)=m\,\Rightarrow\,M_1=0$: Because of Proposition \ref{prop:ranks_matrices} (see the appendix), we can equivalently show that $\Rank (\partial_{y_{[-R_1]}}g(F_x,F_u))=m$ implies $M_1=0$. 
        For arbitrary $\beta\geq0$, the $\beta$th backward-shift of the 1-forms $\mathrm{d}g(F_x,F_u)$ has the form 
        \begin{equation}\label{eq:differential_backward}
        \resizebox{.99 \hsize}{!}{
        $\displaystyle
    		\delta^{-\beta} \mathrm{d} g\left(F_{x}, F_u\right)=\delta^{-\beta}(\partial_{y_{[-R_1-\beta]}} g\left(F_{x}, F_u\right)) \mathrm{d} y_{[-R_1-\beta]}+\ldots\,,
        $}
    	\end{equation}
        where the coefficients of the differentials $\mathrm{d} y_{[-R_1-\beta]}$ of the highest occurring backward-shifts of the flat output are given by the elements of the $\beta$th backward-shift of the matrix $\partial_{y_{[-R_1]}} g(F_x,F_u)$. Thus, if $\partial_{y_{[-R_1]}} g(F_x,F_u)$ is regular (see footnote 4), then for every $\beta \ge 0$ there does not exist a linear combination of the 1-forms $\delta^{-\beta} \mathrm{d} g(F_x,F_u)$ where the differentials $\mathrm{d} y_{[-R_1 - \beta]}$ cancel out. Hence, \eqref{eq:differentials} must hold also without 1-forms \eqref{eq:differential_backward} with $\beta\geq1$, and $M_1=0$ follows.
        
        Direction $M_1=0\,\Rightarrow\,\Rank (\partial_{y_{[-R_1]}}F_x)=m$: With the regular input transformation from Lemma \ref{lem:Jacobian_transformed}, the Jacobian matrix \eqref{eq:Jacobian1} is of the form \eqref{eq:Jacobian_transformed}.
        For $M_1=0$, the $m$ differentials $\mathrm{d}y_{[-R_1]}$ of the left-hand side of \eqref{eq:differentials} then only appear in the $m$ 1-forms $\mathrm{d}F_x$ of the right-hand side of \eqref{eq:differentials}, with their coefficients given by the elements of the Jacobian $\partial_{y_{[-R_1]}}F_x$. Therefore, \eqref{eq:differentials} can only hold if $\Rank (\partial_{y_{[-R_1]}}F_x)=m$. The fact that the above-mentioned input transformation does not affect the parameterization $F_x$ of the state variables completes the proof.
    \end{proof}
    
    The following Corollary is an immediate consequence of Proposition \ref{prop:differentials}.
    \begin{cor}\label{cor:flatness_rank}
        Consider a system \eqref{eq:sysEq} that is flat according to Definition \ref{def:flatness} with a flat output \eqref{eq:flat_output} and the corresponding parameterizing map \eqref{eq:flat_param}. For the special cases of forward- and backward-flatness, the following applies:
        \begin{enumerate}
            \item If the system is forward-flat according to Definition \ref{def:forward-flatness} with a forward-flat output \eqref{eq:forward-flat_output}, then $\Rank (\partial_{y_{[-R_1]}}F_x)=m$ with $R_1=0$.
            \item If the system is backward-flat according to Definition \ref{def:backward-flatness} with a backward-flat output \eqref{eq:backward-flat_output}, then $\Rank (\partial_{y_{[R_2]}}F_u)=m$ with $R_2=0$.
        \end{enumerate}
    \end{cor}

    Furthermore, from the proof of Proposition \ref{prop:differentials}, the following three additional important conclusions can be drawn.\\
    First, consider the forward-flat case according to Definition \ref{def:forward-flatness} with a forward-flat output \eqref{eq:forward-flat_output}, where \eqref{eq:differentials} reduces to 
    \begin{equation*}
            \Span \{\mathrm{d}y_{[0,R_2]}\} \subset \Span \{\mathrm{d}F_x,\mathrm{d}F_u,\delta\mathrm{d}F_u,\ldots,\delta^{M_2}\mathrm{d}F_u\}
    \end{equation*}
    with $R_1=0$ and $M_1=0$. If $M_2>0$, then the matrix $\partial_{y_{[R_2]}}F_u$ is necessarily singular.\footnote{$M_2>0$ means that at least one element of $M_2$ is not zero.}\\
    Second, in the backward-flat case according to Definition \ref{def:backward-flatness} with a backward-flat output \eqref{eq:backward-flat_output}, where \eqref{eq:differentials} reduces to 
    \begin{equation*}
    \begin{split}
        \Span \{\mathrm{d}y_{[-R_1,0]}\} \subset & \Span \{\delta^{-M_1}\mathrm{d}g(F_x,F_u),\ldots,\\
        &\delta^{-1}\mathrm{d}g(F_x,F_u),\mathrm{d}F_x,\mathrm{d}F_u\}
    \end{split}
    \end{equation*}
    with $R_2=0$ and $M_2=0$. If $M_1>0$, then the matrix $\partial_{y_{[-R_1]}}F_x$ cannot have full rank.\\
    If in the general case according to Definition \ref{def:flatness} both $M_1>0$ and $M_2>0$, it immediately follows that both matrices $\partial_{y_{[R_2]}}F_u$ and $\partial_{y_{[-R_1]}}F_x$ cannot have full rank. These rank defects prevent an exact linearization via static feedback. However, the location of the matrices $\partial_{y_{[R_2]}}F_u$ and $\partial_{y_{[-R_1]}}F_x$ in the Jacobian \eqref{eq:Jacobian1} suggests that these rank defects can be resolved by a particularly simple type of dynamic feedback. In the following section, this will be analyzed in detail for systems with two inputs and flat outputs of the form $y=\varphi(x,u)$, which we refer to as $(x,u)$-flat outputs.


    \section{LINEARIZATION OF TWO-INPUT SYSTEMS}\label{sec:linearization}
    In this section, we consider systems of the form \eqref{eq:sysEq} with two inputs and an $(x,u)$-flat output.
    We will show that they can be exactly linearized by a dynamic feedback consisting of shift-chains, which are attached either in the form of forward-shifts to a suitably chosen transformed input or in the form of backward-shifts to a suitably chosen function of the system variables. These two kinds of dynamic feedback will be referred to as prolongations and prelongations, respectively.
    

	\subsection{Linearization of two-input Forward-Flat Systems}
    In this section, we prove that forward-flat systems according to Definition \ref{def:forward-flatness} with $m=2$ and an $(x,u)$-flat output become exactly linearizable by static feedback via prolongations of a suitably chosen transformed input.
	\begin{prop}\label{prop:prolongation}
        Every two-input forward-flat system \eqref{eq:sysEq} with a forward-flat output of the form $y^j=\varphi^j(x,u)$ can be rendered static feedback linearizable by a $d$-fold ($d=\# R_2-n$) prolongation
		\begin{equation*}
			\begin{aligned}
				\Sigma^{(d)}: x^{i,+} & =f^i(x, \hat{\Phi}_u(x, \bar{u})) \\
				\bar{u}^{1,{+}} & =\bar{u}_{[1]}^1 \\
				& \vdots \\
				\bar{u}^{1,{+}}_{[d-1]} & =\bar{u}_{[d]}^1
			\end{aligned}
		\end{equation*}
		with the state $[x^1, \ldots, x^n, \bar{u}^1, \ldots, \bar{u}_{[d-1]}^1]^{T}$, the input $(\bar{u}^1_{[d]},\bar{u}^2)$, and an invertible input transformation $u=\hat{\Phi}_u(x,\bar{u})$.
	\end{prop} 
	
	\begin{proof}
		The proof follows the same argumentation as the one in \cite{GstottnerKolarSchoberl:2022}. Let $y^j=\varphi^j(x,u)$ be a flat output of \eqref{eq:sysEq}, $r_{2,1},\; r_{2,2}>0$ be unique minimal integers of \eqref{eq:flat_param_forward} and $\rho_{1},\, \rho_2 > 0$ be the relative degrees of the components of this flat output.
		After applying the invertible input transformation $\bar{u}^1=\varphi^1_{[\rho_1]}(x, u),\; \bar{u}^2=u^2$ (permute $u^1$ and $u^2$ if necessary), the forward-shifts of the components $y^1,\, y^2$ of the flat output up to the orders $r_{2,1},\; r_{2,2}$ have the form\footnote{If $\rho_j=0$, then the functions $\varphi^j_{[0,\rho_j-1]}$ in \eqref{eq:proof_prop_prolonged} are not present.}
		\begin{equation}\label{eq:proof_prop_prolonged}
        \resizebox{.99 \hsize}{!}{
        $\displaystyle
			\begin{array}{rlrl}
				y^1 \hspace{-2ex}& =\varphi^1(x) & y^2 \hspace{-2ex}&=\varphi^2(x) \\
				& \vdots & & \vdots \\
				y_{[\rho_1-1]}^1 \hspace{-2ex}& =\varphi^1_{[\rho_1-1]}(x) & y_{[\rho_2-1]}^2 \hspace{-2ex}& =\varphi^2_{[\rho_2-1]}(x) \\
				y_{[\rho_1]}^1 \hspace{-2ex}& =\varphi^1_{[\rho_1]}(x,u) =\bar{u}^1 & y_{[\rho_2]}^2 \hspace{-2ex}& =\varphi^2_{[\rho_2]}(x, \bar{u}^1) \\
				y_{[\rho_1+1]}^1 \hspace{-2ex}& =\bar{u}^1_{[1]} & y_{[\rho_2+1]}^2 \hspace{-2ex}& =\varphi^2_{[\rho_2+1]}(x, \bar{u}^1, \bar{u}^1_{[1]}) \\
				& \vdots & & \vdots \\
				y_{[r_{2,1}]}^1 \hspace{-2ex}& =\bar{u}^1_{[r_{2,1}-\rho_1]} & y_{[r_{2,2}]}^2 \hspace{-2ex}& =\varphi^2_{[r_{2,2}]}(x, \bar{u}^1, \bar{u}^1_{[1]},\ldots\\
                &  &  & \hspace{1.2cm}\bar{u}^1_{[r_{2,2}-\rho_2]}, \bar{u}_2)\,.
			\end{array}
        $}
		\end{equation}
        While it is obvious that the forward-shifts of $y^1$ must be of this form, we will also prove it for the forward-shifts of the second component $y^2$. By definition of the relative degree, $\varphi_{[\rho_2]}^2$ depends explicitly on at least one of the inputs. The forward-shifts $\varphi_{\left[\rho_2\right]}^2, \ldots, \varphi_{\left[r_{2,2}-1\right]}^2$ are independent of $\bar{u}^2$ and its forward-shifts due to the following. If $\bar{u}^2$ would occur in $\varphi_{[s]}^2$ with $\rho_2 \leq s<r_{2,2}$, then $\varphi_{[s+1]}^2, \ldots, \varphi_{\left[r_{2,2}\right]}^2$ would depend on forward-shifts of $\bar{u}^2$ and hence be useless for constructing functions of $x$ and $\bar{u}$ only. On the other hand, $\bar{u}^2$ must explicitly occur in $\varphi_{\left[r_{2,2}\right]}^2$, since otherwise $\bar{u}^2$ could not be expressed by $y^1, \ldots, y^1_{[r_{2,1}]}, y^2, \ldots, y_{[r_{2,2}]}^2$.
        Moreover, due to the minimality of $r_{2,1}$ and the explicit dependence of $\varphi^2_{[r_{2,2}]}$ on $\bar{u}^1_{[r_{2,2}-\rho_2]}$, it follows that $r_{2,1}-\rho_1=r_{2,2}-\rho_2$.\\
        Now it is only left to prove that \eqref{eq:proof_prop_prolonged} describes a diffeomorphism. The forward-shifts of a flat output up to an arbitrary order are functionally independent. Thus, we only have to prove that the $n+m+\left(r_{2,1}-\rho_1\right)$ variables $x, \bar{u}^1, \ldots, \bar{u}_{[r_{2,1}-\rho_1]}^1, \bar{u}^2$ can be expressed by the $r_{2,1}+r_{2,2}+m$ forward-shifts $y^1, \ldots, y_{[r_{2,1}]}^1, y^2, \ldots, y_{[r_{2,2}]}^2$, i.e.
        \begin{equation}\label{eq:dim_argument_forward}
			r_{2,1}+r_{2,2}+m=n+m+r_{2,1}-\rho_1 \, \rightarrow \, \rho_1+r_{2,2}=n\,.
		\end{equation}
		This follows immediately since it must be possible to construct exactly $n$ independent functions of $x$ only from \eqref{eq:proof_prop_prolonged}.\\
		The diffeomorphism \eqref{eq:proof_prop_prolonged} is just the inverse of the parameterizing map of the prolonged system
		\begin{equation*}
			\begin{aligned}
				x^{i,+} & =f^i(x, \bar{u}) \\
				\bar{u}^{1,{+}} & =\bar{u}_{[1]}^1 \\
				& \vdots \\
				\bar{u}^{1,{+}}_{[r_{2,1}-\rho_1-1]} & =\bar{u}_{[r_{2,1}-\rho_1]}^1
			\end{aligned}
		\end{equation*}
        with the input $(\bar{u}^1_{[r_{2,1}-\rho_1]},\bar{u}^2)$ and $d=\# R_2-n=r_{2,1}-\rho_1$ from \eqref{eq:dim_argument_forward}. Because the parameterizing map of this prolonged system is a diffeomorphism, it is static feedback linearizable.
	\end{proof}

    \begin{rem}
        We would like to point out that, according to Theorem 12 in \cite{KolarDiwoldSchoberl:2021}, every forward-flat system --- as defined in Definition \ref{def:forward-flatness} --- has a flat output of the form $y^j=\varphi^j(x)$. 
        Therefore, Proposition \ref{prop:prolongation} applies to every two-input forward-flat system, whereas the linearization via prolongations of a suitably chosen input has so far only been proven for $(x,u)$-flat systems in the continuous-time case.
    \end{rem}

    With the following practical example of the well known planar VTOL, for which a forward-flat discrete-time model was derived in \cite{KolarSchrotshamerSchoberl:2025}, we validate the results from Proposition \ref{prop:prolongation}.

    \begin{example}
        The planar VTOL aircraft is a popular example w.r.t. flatness analysis and the calculation of flat outputs in the continuous-time literature, as e.g. in \cite{FliessLevineMartinRouchon:1999}, \cite{SchoeberlRiegerSchlacher:2010}, and \cite{SchoeberlSchlacher:2011}. It has been shown that a flat discrete-time system
        \begin{equation}
			\begin{aligned}
				x^{1,+} & = x^1+T_sx^3 \\
				x^{2,+} & = x^2+T_sx^4 \\
				x^{3,+} & = x^3+T_s\sin(x^5)(\varepsilon(x^6)^2-u^1) \\
				x^{4,+} & = x^4+T_s\cos(x^5)(-\varepsilon(x^6)^2+u^1)-g T_s \\
				x^{5,+} & = x^5+T_sx^6 \\
				x^{6,+} & = x^6+T_su^2
			\end{aligned}
			\label{eq:sys_vtol}
		\end{equation}
        with $n=6$, $m=2$ can be obtained by combining an Euler discretization for some sampling time $T_{s}>0$ with a suitable prior state transformation. 
        The discrete-time system \eqref{eq:sys_vtol} is forward-flat with an $x$-flat output $y=(x^{1},x^{2})$ because the sequence of distributions derived according to Algorithm 1 of \cite{KolarDiwoldSchoberl:2023} satisfies Theorem 4 of \cite{KolarDiwoldSchoberl:2023}.\\
		The forward-shifts of the components of the flat output $y^j=\varphi^j(x)=(x^1,x^2)$ up to the orders $r_{2,1},\; r_{2,2}$ are as follows
        \begin{equation*}\label{eq:vtol_y}
        \resizebox{1.00 \hsize}{!}{
        $\displaystyle
			\begin{array}{rlrl}
				y^1\hspace{-2ex}&=x^1 & y^2\hspace{-2ex} & =x^2 \\
				y_{[1]}^1\hspace{-2ex} &=x^1+T_sx^3 & y_{[1]}^2\hspace{-2ex} &=x^2+T_sx^4 \\
				y_{[2]}^1\hspace{-2ex} &=\bar{u}^1 & y_{[2]}^2\hspace{-2ex} &=\varphi^2_{[2]}(x^1,x^2,x^3,x^4,x^5,\bar{u}^1) \\
				y_{[3]}^1\hspace{-2ex} &=\bar{u}^1_{[1]} & y_{[3]}^2\hspace{-2ex} &=\varphi^2_{[3]}(x^1,x^2,x^3,x^4,x^5,x^6,\bar{u}^1,\bar{u}^1_{[1]}) \\
				y_{[4]}^1\hspace{-2ex} &=\bar{u}^1_{[2]} & y_{[4]}^2\hspace{-2ex} &=\varphi^2_{[4]}(x^1,x^2,x^3,x^4,x^5,x^6,\bar{u}^1,\bar{u}^1_{[1]},\\
                &&&\hspace{4,1cm}\bar{u}^1_{[2]},\bar{u}^2) 
			\end{array}
            $}
		\end{equation*}
		with $\bar{u}^1=\varphi^1_{[2]}(x^1,x^3,x^5,x^6,u^1)$. The parameterization of the system variables is of the form $x^i=F_x^i(y^1_{[0,3]},y^2_{[0,3]})$, $u^j=F_u^j(y^1_{[0,4]},y^2_{[0,4]})$ and the difference $d$ is given by $d=\# R_2-n=r_{2,1}-\rho_1=2$.\\
		According to Proposition \ref{prop:prolongation}, the system \eqref{eq:sys_vtol} together with the invertible input transformation $\bar{u}^1=\varphi^1_{[2]}(x^1,x^3,x^5,x^6,u^1),\,\bar{u}^2=u^2$ and a $d=2$-fold prolongation, i.e.
		\begin{equation*}\label{eq:ext_vtol}
			\begin{array}{rl}
				x^{1,+}\hspace{-2ex} & = x^1+T_sx^3 \\
				x^{2,+}\hspace{-2ex} &= x^2+T_sx^4 \\
				x^{3,+}\hspace{-2ex} &=\frac{-T_s x^3 - x^1 + \bar{u}^1}{T_s} \\
				x^{4,+}\hspace{-2ex} &= \frac{(-2 T_s x^3 - x^1 + \bar{u}^1)\cot(x^5) - T_s^2 g +T_s x^4}{T_s} \\
				x^{5,+}\hspace{-2ex} &= x^5+T_sx^6 \\
				x^{6,+}\hspace{-2ex} &= x^6+T_s\bar{u}^2 \\
				\bar{u}^{1,+}\hspace{-2ex} & = \bar{u}^1_{[1]} \\
				\bar{u}^{1,+}_{[1]}\hspace{-2ex} & = \bar{u}^1_{[2]}
			\end{array}
		\end{equation*}
		with the $n=6+2=8$-dimensional state $[x^1,\ldots,x^6,\bar{u}^1,\bar{u}^1_{[1]}]^T$ and new input $(\bar{u}^1_{[2]},\bar{u}^2)$ is static feedback linearizable. This can be verified using the static feedback linearization test in \cite{Grizzle:1986}.
	\end{example}

	\subsection{Linearization of two-input Backward-Flat Systems}
    In this section, we prove that backward-flat systems according to Definition \ref{def:backward-flatness} with $m=2$ and an $(x,u)$-flat output become exactly linearizable by static feedback via prelongations of a suitably chosen function of the system variables $x$ and $u$.
	\begin{prop}\label{prop:prelongation}
		Every two-input backward-flat system \eqref{eq:sysEq} with a backward-flat output of the form $y^j=\varphi^j(x,u)$ and $\Rank (\partial_u \varphi^j)=m$ can be rendered static feedback linearizable by a $d$-fold ($d=\# R_1-n$) prelongation
		\begin{equation*}
			\begin{aligned}
				\Sigma^{(d)}: \bar{\zeta}^{1,{+}}_{[-d]} & =\bar{\zeta}^{1}_{[-d+1]} \\
				& \vdots \\
				\bar{\zeta}^{1,+}_{[-1]} & =\bar{g}^1(x,u)\\
				x^{i,+} & =f^i(x, u)
			\end{aligned}
		\end{equation*}
		with the state $[\bar{\zeta}^{1}_{[-d]}, \ldots, \bar{\zeta}^{1}_{[-1]}, x^1, \ldots, x^n]^{T}$, the input $(u^1, u^2)$ and an invertible transformation $\zeta_{[-1]}=\hat{\Phi}_{\zeta}(x,\bar{\zeta}_{[-1]})$.
	\end{prop} 

	\begin{proof}
		Let $y^j=\varphi^j(x,u)$ be a flat output of \eqref{eq:sysEq}, $r_{1,1}\,,\,r_{1,2}>0$ be unique minimal integers of \eqref{eq:flat_param_back} and $\gamma_1,\,\gamma_2>0$ be unique integers such that $y^j_{[-\gamma_j]}$ are the minimal order backward-shifts where $\zeta^j_{[-1]}$ appear for the first time, i.e. $y_{[-\gamma_j]}^j=\varphi^j_{[-\gamma_j]}(\zeta_{[-1]},x)$.
		After applying the invertible transformation $\bar{\zeta}^1_{[-1]}=\varphi^1_{[-\gamma_1]}(\zeta_{[-1]},x),\; \bar{\zeta}^2_{[-1]}=\zeta^2_{[-1]}$ (permute $\zeta^1$ and $\zeta^2$ if necessary), the backward-shifts of the components $y^1,\, y^2$ of the flat output up to the orders $r_{1,1},\, r_{1,2}$ have the form
        \begin{equation}\label{eq:proof_prop_prelonged}
        \resizebox{.99 \hsize}{!}{
        $\displaystyle
			\begin{array}{rlrl}
				y_{[-r_{1,1}]}^1 \hspace{-2ex}& =\bar{\zeta}^1_{[-r_{1,1}+\gamma_1-1]} & y_{[-r_{1,2}]}^2 \hspace{-2ex}&=\varphi^2_{[-r_{1,2}]}(\bar{\zeta}^1_{[-r_{1,2}+\gamma_2-1]}, \\
                & & & \hspace{1.3cm}\ldots,\bar{\zeta}^1_{[-1]},x) \\
				& \vdots & & \vdots \\
				y_{[-\gamma_1]}^1 \hspace{-2ex}& =\varphi^1_{[-\gamma_1]}(\zeta_{[-1]},x) & y_{[-\gamma_2]}^2 \hspace{-2ex}& =\varphi^2_{[-\gamma_2]}(\bar{\zeta}^1_{[-1]},x) \\
                & =\bar{\zeta}^1_{[-1]} & & \\
				y_{[-\gamma_1+1]}^1 \hspace{-2ex}& =\varphi^1_{[-\gamma_1+1]}(x) & y_{[-\gamma_2+1]}^2 \hspace{-2ex}& =\varphi^2_{[-\gamma_2+1]}(x) \\
				& \vdots & & \vdots \\
				y_{[-1]}^1 \hspace{-2ex}& =\varphi^1_{[-1]}(x) & y_{[-1]}^2 \hspace{-2ex}& =\varphi^2_{[-1]}(x)\\
				y^1 \hspace{-2ex}& =\varphi^1(x,u) & y^2 \hspace{-2ex}& =\varphi^2(x, u) \,.
			\end{array}
        $}
		\end{equation}
		First, because of \eqref{eq:sysEq_ext_inv}, the backward-shifts $\varphi_{[-r_{1,j}]}^j, \ldots, \varphi^j_{[-1]}$ are independent of $u$.
		On the other hand, $\partial_{u} \varphi^j=m$ must hold, since otherwise $u^1,\,u^2$ could not be expressed by $y^1_{\left[-r_{1,1}\right]}, \ldots, y^1, y^2_{[-r_{1,2}]}, \ldots, y^2$.
		The backward-shifts $\varphi^2_{\left[-r_{1,2}\right]}, \ldots, \varphi^2_{\left[-\gamma_2\right]}$ are independent of $\bar{\zeta}^2_{[-1]}$ and its backward-shifts due to the following. By definition $\varphi_{[-\gamma_2]}^2$ depends explicitly on at least one of the backward-shifts $\zeta^j_{[-1]}$. If $\bar{\zeta}^2_{[-1]}$ would occur in $\varphi_{[-s]}^2$ with $\gamma_2 \leq s \leq r_{1,2}$, then $\varphi_{[-s-1]}^2, \ldots, \varphi_{\left[-r_{1,2}\right]}^2$ would depend on backward-shifts of $\bar{\zeta}^2_{[-1]}$ and hence be useless for constructing functions of $x$ only.
        Again, due to the minimality of $r_{1,1}$ and the explicit dependence of $\varphi^2_{[-r_{1,2}]}$ on $\bar{\zeta}^1_{[-r_{1,2}+\gamma_2-1]}$, it follows that $r_{1,1}-\gamma_1=r_{1,2}-\gamma_2$.\\
        Next, we prove that \eqref{eq:proof_prop_prelonged} describes a diffeomorphism. With the backward-shifts of a flat output up to an arbitrary order being functionally independent, we only have to prove that the $n+m+r_{1,1}-\gamma_1+1$ variables $\bar{\zeta}^1_{[-r_{1,1}+\gamma_1-1,-1]}, x, u$ can be expressed by the $r_{1,1}+r_{1,2}+m$ backward-shifts $y^1_{[-r_{1,1}]}, \ldots, y^1, y^2_{[-r_{1,2}]}, \ldots, y^2$, i.e.
        \begin{equation}\label{eq:dim_argument_backward}
        \resizebox{.99 \hsize}{!}{
        $\displaystyle
			r_{1,1}+r_{1,2}+m=n+m+r_{1,1}-\gamma_1+1 \, \rightarrow \, r_{1,2}+\gamma_1-1=n\,.
        $}
		\end{equation}
		This follows immediately since it must be possible to construct exactly $n$ independent functions of $x$ only from \eqref{eq:proof_prop_prelonged}.\\
        Again, the diffeomorphism \eqref{eq:proof_prop_prelonged} is the inverse of the parameterizing map of the prelonged system
		\begin{equation*}
			\begin{aligned}
				\bar{\zeta}^{1,{+}}_{[-r_{1,1}+\gamma_1-1]} & =\bar{\zeta}^{1}_{[-r_{1,1}+\gamma_1]} \\
				& \vdots \\
				\bar{\zeta}^{1,+}_{[-1]} & =\bar{g}^1(x,u)\\
				x^{i,+} & =f^i(x, u)
			\end{aligned}
		\end{equation*}
		with $\bar{g}^1(x,u)=\varphi^1_{[-\gamma_1+1]}(x,u)$ and $d=\# R_1-n=r_{1,1}-\gamma_1+1$ from \eqref{eq:dim_argument_backward}. Therefore, the prelonged system is static feedback linearizable, because the corresponding parameterizing map is also a diffeomorphism.
	\end{proof}

    Next, we illustrate the results of Proposition \ref{prop:prelongation} by applying it to an academic example.
	\begin{example}
		Consider the system
		\begin{equation}
			\begin{aligned}
				x^{1,+} & =x^1+x^4\\
				x^{2,+} & =x^2+u^2\\
				x^{3,+} & =x^3+x^4 u^2\\
				x^{4,+} & =u^1\\
				x^{5,+} & =u^2
			\end{aligned}
			\label{eq:sysEq_academic}
		\end{equation}
		with $n=5$, $m=2$, which is neither forward-flat nor static feedback linearizable. By applying map \eqref{eq:sysEq_ext} with $g^1(x,u)=x^1,\,g^2(x,u)=x^5$, the backward-shifts of the components of the flat output $y^j=\varphi^j(x,u)=(x^1+x^4+u^1,x^3+x^4u^2-x^2u^1-u^1u^2)$ up to the orders $r_{1,1},\; r_{1,2}$ follow as
        \begin{equation*}\label{eq:academic_y}
        \resizebox{0.99 \hsize}{!}{
        $\displaystyle
			\begin{array}{rlrl}
				y_{[-4]}^1\hspace{-2ex} &=\bar{\zeta}^1_{[-2]} & y_{[-3]}^2\hspace{-2ex} & =x^3-x^1x^5+\bar{\zeta}^1_{[-1]}(2x^5 \\
				&&&\quad-x^2)+\bar{\zeta}^1_{[-2]}(x^2-x^5) \\
                y_{[-3]}^1\hspace{-2ex} &=\bar{\zeta}^1_{[-1]} & y_{[-2]}^2\hspace{-2ex} &=x^3-x^2(x^1-\bar{\zeta}^1_{[-1]}) \\                
				y_{[-2]}^1\hspace{-2ex} &=x^1 & y_{[-1]}^2\hspace{-2ex} &=x^3-x^2x^4 \\
				y_{[-1]}^1\hspace{-2ex} &=x^1+x^4 & y^2\hspace{-2ex} &=x^3+x^4u^2-x^2u^1-u^1u^2\\
				y^1 \hspace{-2ex} &=x^1+x^4+u^1 
			\end{array}
            $}
		\end{equation*}
        with $\bar{\zeta}^1_{[-1]}=\zeta^1_{[-1]}$ and $\bar{g}^1(x,u)=g^1(x,u)=x^1$.
		The parameterization of the system variables is of the form $x^i=F_x^i(y^1_{[-4,-1]},y^2_{[-3,-1]})$, $u^j=F_u^j(y^1_{[-4,0]},y^2_{[-3,0]})$ and the difference $d$ is given by $d=\# R_1-n=r_{1,1}-\gamma_1+1=2$.
        Since the (backward)-flat output and the parameterizing map are of the form \eqref{eq:backward-flat_output} and \eqref{eq:flat_param_back}, respectively, the system \eqref{eq:sysEq_academic} belongs to the class of backward-flat systems as defined in Definition \ref{def:backward-flatness}.
        With the given (backward)-flat output that only depends on $x$, $u$ (and backward-shifts $\zeta_{[-Q_1,-1]}$) and because the parameterization is of the form \eqref{eq:flat_param_back}, the system \eqref{eq:sysEq_academic} indeed belongs to the class of backward-flat systems. \\
		Now, the system \eqref{eq:sysEq_academic} together with the invertible transformation $\bar{\zeta}^1_{[-1]}=\varphi^1_{[-3]}(\zeta^1_{[-1]})=\zeta^1_{[-1]},\,\bar{\zeta}^2_{[-1]}=\zeta^2_{[-1]}$ and a $d=2$-fold prelongation, i.e.
		\begin{equation*}
			\begin{aligned}
                \bar{\zeta}^{1,+}_{[-2]} & =\bar{\zeta}^1_{[-1]}\\
				\bar{\zeta}^{1,+}_{[-1]} & =x^1\\
				x^{1,+} & =x^1+x^4\\
				x^{2,+} & =x^2+u^2\\
				x^{3,+} & =x^3+x^4 u^2\\
				x^{4,+} & =u^1\\
				x^{5,+} & =u^2
			\end{aligned}
			\label{eq:sys_ext4}
		\end{equation*}
		with the $n=5+2=7$-dimensional state $[\bar{\zeta}^1_{[-2]},\bar{\zeta}^1_{[-1]},x^1,\ldots,x^5]^T$ and input $u$ is static feedback linearizable according to Proposition \ref{prop:prelongation}. This can once again be verified using the static feedback linearization test in \cite{Grizzle:1986}.
	\end{example}

	\subsection{Linearization of two-input Flat Systems}
    With these two special cases, we can now prove that flat systems according to Definition \ref{def:flatness} with $m=2$ and an $(x,u)$-flat output become exactly linearizable by static feedback via a combination prolongations of a suitably chosen transformed input and prelongations of a suitably chosen function of the system variables $x$ and $u$.
	\begin{prop}\label{prop:pre_prolongation}
		Every two-input flat system \eqref{eq:sysEq} with a flat output of the form $y^j=\varphi^j(x,u)$ can be rendered static feedback linearizable by a $d_1$-fold prelongation and $d_2$-fold prolongation ($d=d_1+d_2=\# R-n$)
		\begin{equation*}
			\begin{aligned}
				\Sigma^{(d)}: \bar{\zeta}^{1,{+}}_{[-d_1]} & =\bar{\zeta}^{1}_{[-d_1+1]} \\
				& \vdots \\
				\bar{\zeta}^{1,+}_{[-1]} & =\bar{g}^{1}(x,\hat{\Phi}_u(x, \bar{u}))\\
				x^{i,+} &=f^i(x, \hat{\Phi}_u(x, \bar{u})) \\
				\bar{u}^{1,{+}} & =\bar{u}_{[1]}^1 \\
				& \vdots \\
				\bar{u}^{1,{+}}_{[d_2-1]} & =\bar{u}_{[d_2]}^1
			\end{aligned}
		\end{equation*}
		with the state $[ \bar{\zeta}^{1}_{[-d_1]}, \ldots, \bar{\zeta}^{1}_{[-1]}, x^1, \ldots, x^n,\bar{u}^1, \ldots, \bar{u}_{[d_2-1]}^1]^{T}$, the input $(\bar{u}^2, \bar{u}^1_{[d_2]})$ and invertible transformations $\zeta_{[-1]}=\hat{\Phi}_{\zeta}(x,\bar{\zeta}_{[-1]})$ and $u=\hat{\Phi}_u(x,\bar{u})$.
	\end{prop} 
	
	\begin{proof}
		Let $y^j=\varphi^j(x,u)$ be a flat output of \eqref{eq:sysEq}, $r_{1,1},\,r_{1,2},\,r_{2,1},\,r_{2,2}>0$ be unique minimal integers of \eqref{eq:flat_param}, $\rho_{1},\, \rho_2 \ge 0$ be the relative degrees of the components of the flat output and $\gamma_1,\,\gamma_2>0$ be unique minimal integers such that $y_{[-\gamma_j]}^j=\varphi^j_{[-\gamma_j]}(\zeta_{[-1]},x)$.
		After applying the invertible transformation $\bar{\zeta}^1_{[-1]}=\varphi^1_{[-\gamma_1]}(\zeta_{[-1]},x),\; \bar{\zeta}^2_{[-1]}=\zeta^2_{[-1]}$ (permute $\zeta^1$ and $\zeta^2$ if necessary) and the invertible input transformation $\bar{u}^1=\varphi^1_{[\rho_1]}(x, u),\; \bar{u}^2=u^2$ (permute $u^1$ and $u^2$ if necessary) the backward- and forward-shifts of the components $y^1,\, y^2$ of the flat output up to the orders $r_{1,1},\; r_{1,2}$ and $r_{2,1},\; r_{2,2}$ have the form\footnote{If $\rho_j=0$, then the functions $\varphi^j_{[0,\rho_j-1]}$ in \eqref{eq:proof_prop_flat} are not present.}
		
		\begin{equation}\label{eq:proof_prop_flat}
        \resizebox{.99 \hsize}{!}{
        $\displaystyle
			\begin{array}{rlrl}
				y_{[-r_{1,1}]}^1 \hspace{-2ex}& =\bar{\zeta}^1_{[-r_{1,1}+\gamma_1-1]} & y_{[-r_{1,2}]}^2 \hspace{-2ex}&=\varphi^2_{[-r_{1,2}]}(\bar{\zeta}^1_{[-r_{1,2}+\gamma_2-1],} \\
                &&&\hspace{1.6cm} \ldots,\bar{\zeta}^1_{[-1]},x)\\
				& \vdots & & \vdots \\
				y_{[-\gamma_1]}^1 \hspace{-2ex}& =\varphi^1_{[-\gamma_1]}(\zeta_{[-1]},x) & y_{[-\gamma_2]}^2 \hspace{-2ex}& =\varphi^2_{[-\gamma_2]}(\bar{\zeta}^1_{[-1]},x) \\
                &=\bar{\zeta}^1_{[-1]}&&\\
				y_{[-\gamma_1+1]}^1 \hspace{-2ex}& =\varphi^1_{[-\gamma_1+1]}(x) & y_{[-\gamma_2+1]}^2 \hspace{-2ex}& =\varphi^2_{[-\gamma_2+1]}(x) \\
				& \vdots & & \vdots \\
				y^1 \hspace{-2ex}& =\varphi^1(x) & y^2 \hspace{-2ex}& =\varphi^2(x) \\
				& \vdots & & \vdots \\
				y_{[\rho_1-1]}^1 \hspace{-2ex}& =\varphi^1_{[\rho_1-1]}(x) & y_{[\rho_2-1]}^2 \hspace{-2ex}& =\varphi^2_{[\rho_2-1]}(x) \\
				y_{[\rho_1]}^1 \hspace{-2ex}& =\varphi^1_{[\rho_1]}(x,u) =\bar{u}^1 & y_{[\rho_2]}^2 \hspace{-2ex}& =\varphi^2_{[\rho_2]}(x, \bar{u}^1) \\
				& \vdots & & \vdots \\
				y_{[r_{2,1}]}^1 \hspace{-2ex}& =\bar{u}^1_{[r_{2,1}-\rho_1]} & y_{[r_{2,2}]}^2 \hspace{-2ex} & =\varphi^2_{[r_{2,2}]}(x, \bar{u}^1, \bar{u}^1_{[1]},\\
                &&&\hspace{0.1cm} \ldots,\bar{u}^1_{[r_{2,2}-\rho_2]}, \bar{u}_2)\,.
			\end{array}
        $}
		\end{equation}
		The reasoning behind why \eqref{eq:proof_prop_flat} must be of this form can be transferred from the proofs for Propositions \ref{prop:prolongation} and \ref{prop:prelongation}, with the difference that $\Rank (\partial_u\varphi^j)=m$ no longer has to apply. Therefore, $r_{1,1}-\gamma_1=r_{1,2}-\gamma_2$ and $r_{2,1}-\rho_1=r_{2,2}-\rho_2$ hold true as well.\\
		We can again prove that \eqref{eq:proof_prop_flat} describes a diffeomorphism due to the following. As the forward- and backward-shifts of a flat output up to an arbitrary order are functionally independent, we only have to show that the $r_{1,1}+1-\gamma_1+n+m+r_{2,1}-\rho_1$ variables $\bar{\zeta}^1_{[-r_{1,1}+\gamma_1-1,-1]}, x, \bar{u}^1, \bar{u}_{[1,r_{2,1}-\rho_1]}^1, \bar{u}^2$ can be expressed by the $\# R+m$ backward- and forward-shifts $y_{[-r_{1,1},r_{2,1}]}^1, y_{[-r_{1,2},r_{2,2}]}^2$, i.e.
		\begin{equation}\label{eq:dim_argument3}
        \begin{split}
            \# R+m&=r_{1,1}+1-\gamma_1+n+m+r_{2,1}-\rho_1\\
            n&=r_{1,2}+r_{2,2}+\gamma_1+\rho_1-1\,.
        \end{split}
		\end{equation}
		This follows immediately since it must be possible to construct exactly $n$ independent functions of $x$ only from \eqref{eq:proof_prop_flat}.\\
        Since the diffeomorphism \eqref{eq:proof_prop_flat} is just the inverse of the parameterizing map of the extended system
        \begin{equation*}
			\begin{aligned}
				\bar{\zeta}^{1,{+}}_{[-r_{1,1}+\gamma_1-1]} & =\bar{\zeta}^{1}_{[-r_{1,1}+\gamma_1]} \\
				& \vdots \\
				\bar{\zeta}^{1,+}_{[-1]} & =\bar{g}^{1}(x,\hat{\Phi}_u(x, \bar{u}))\\
				x^{i,+} & =f^i(x, \hat{\Phi}_u(x, \bar{u})) \\
				\bar{u}^{1,{+}} & =\bar{u}_{[1]}^1 \\
				& \vdots \\
				\bar{u}^{1,{+}}_{[r_{2,1}-\rho_1-1]} & =\bar{u}_{[r_{2,1}-\rho_1]}^1
			\end{aligned}
		\end{equation*}
        with the input $(\bar{u}^1_{[r_{2,1}-\rho_1]},\bar{u}^2)$ and $d=\# R-n=\overbrace{r_{1,1}+1-\gamma_1}^{d_1}+\overbrace{r_{2,1}-\rho_1}^{d_2}$ from \eqref{eq:dim_argument3}, the system is static feedback linearizable.
	\end{proof}

	\begin{example}
		Consider the practical example studied in \cite{Kaldmae:2021} or \cite{Aranda-BricaireMoog:2008} with the system equations
		\begin{equation}
			\begin{aligned}
				x^{1,+} & =x^1+u^1\cos(u^2)\\
				x^{2,+} & =x^2+u^1\sin(u^2)\\
				x^{3,+} & =x^3+u^2
			\end{aligned}
			\label{eq:sysEq_wheeled_robot}
		\end{equation}
		and $n=3$, $m=2$. According to the tests in \cite{KolarDiwoldSchoberl:2023} and \cite{Grizzle:1986}, the system describing an Euler approximation of the continuous-time model of the wheeled mobile robot, with the sampling time of one time unit, is neither forward-flat nor static feedback linearizable. With the map \eqref{eq:sysEq_ext}, for which we choose $g^1(x,u)=x^3,\,g^2(x,u)=x^1$, the backward- and forward-shifts of the components $y^j=\varphi^j(x,u)=(x^3,x^1\sin(u^2)-x^2\cos(u^2))$ of the flat output up to the orders $r_{1,1},\; r_{1,2},\; r_{2,1},\; r_{2,2}$ are as follows
		\begin{equation*}\label{eq:wheeled_robot_y}
			\begin{array}{rl}
				y_{[-1]}^1\hspace{-2ex} &=\bar{\zeta}^1_{[-1]} \\
				y^1\hspace{-2ex} &=x^3 \\
				y_{[1]}^1\hspace{-2ex} &=\bar{u}^1 \\
				y_{[2]}^1\hspace{-2ex} &=\bar{u}^1_{[1]} \\\\
                y_{[-1]}^2\hspace{-2ex} &=x^1\sin(x^3-\bar{\zeta}^1_{[-1]})-x^2\cos(x^3-\bar{\zeta}^1_{[-1]}) \\
                y^2\hspace{-2ex} &=x^1\sin(\bar{u}^1-x^3)-x^2\cos(\bar{u}^1-x^3) \\
                y_{[1]}^2\hspace{-2ex} &=\sin(\bar{u}^1-\bar{u}^1_{[1]})(-x^1-\bar{u}^2\cos(\bar{u}^1-x^3)) \\
                &\quad +\cos(\bar{u}^1-\bar{u}^1_{[1]})(-x^2-\bar{u}^2\sin(\bar{u}^1-x^3)))
			\end{array}
		\end{equation*}
		with $\bar{u}^1=x^3+u^2$, $\bar{\zeta}^1_{[-1]}=\zeta^1_{[-1]}$ and $\bar{g}^1(x,u)=g^1(x,u)=x^3$. The parameterization of the system equations are of the form $x^i=F_x^i(y^1_{[-1,1]},y^2_{[-1,0]})$, $u^j=F_u^j(y^1_{[-1,2]},y^2_{[-1,1]})$, which results in a difference $d=\# R-n=\overbrace{r_{1,1}+1-\gamma_1}^{d_1}+\overbrace{r_{2,1}-\rho_1}^{d_2}=1+1=2$. Since the parameterizing map and the flat output are of the form \eqref{eq:flat_param} and \eqref{eq:flat_output}, respectively, the system \eqref{eq:sysEq_wheeled_robot} belongs to the class of flat systems as defined in Definition \ref{def:flatness}.\\
        As described in Proposition \ref{prop:pre_prolongation} and verified by the static feedback linearization test in \cite{Grizzle:1986}, the system \eqref{eq:sysEq_wheeled_robot} together with the invertible input transformation $\bar{u}^1=\varphi^1_{[1]}(x^3,u^2)=x^3+u^2,\,\bar{u}^2=u^2$, invertible transformation $\bar{\zeta}^1_{[-1]}=\varphi^1_{[-1]}(\zeta^1_{[-1]})=\zeta^1_{[-1]},\,\bar{\zeta}^2_{[-1]}=\zeta^2_{[-1]}$, $d_1=1$-fold prelongation and $d_2=1$-fold prolongation, i.e.
		\begin{equation*}
			\begin{aligned}
				\bar{\zeta}^{1,+}_{[-1]} & =x^3\\
				x^{1,+} & =x^1+\bar{u}^2\cos(\bar{u}^1-x^3)\\
				x^{2,+} & =x^2+\bar{u}^2\sin(\bar{u}^1-x^3)\\
				x^{3,+} & =\bar{u}^1\\
				\bar{u}^{1,+} & =\bar{u}^{1}_{[1]}
			\end{aligned}
			\label{eq:sys_ext5}
		\end{equation*}
		with the $n=3+1+1=5$-dimensional state $[\bar{\zeta}^1_{[-1]},x^1,\ldots,x^3,\bar{u}^1]^T$ and the input $(\bar{u}^1_{[1]},\bar{u}^2)$ is static feedback linearizable.
	\end{example}


\bibliographystyle{IEEEtran}
\bibliography{IEEEabrv,bibliography}

\appendix \label{sec:appendix}


    \begin{prop}\label{prop:ranks_matrices}
        The ranks of the matrices $\partial_{y_{[-R_1]}} g\left(F_{x}, F_u\right)$ and $\partial_{y_{[-R_1]}} F_{x}$ coincide.
    \end{prop}
    \begin{proof}
        Because of the chain rule, the identity 
        \begin{equation*}
            \begin{bmatrix}
                \partial_{y_{[-R_1]}} f(F_x,F_u)\\ \partial_{y_{[-R_1]}} g(F_x,F_u)
            \end{bmatrix}=\left(\partial_{x, u} \begin{bmatrix}
                f\\ g
            \end{bmatrix} \circ F_{x,u}\right) \cdot \begin{bmatrix}
                \partial_{y_{[-R_1]}} F_{x}\\
                \partial_{y_{[-R_1]}} F_{u}
            \end{bmatrix}
        \end{equation*}
        obviously holds, and with $f(F_x,F_u)=\delta(F_x)$ we get
        \begin{equation}\label{eq:rank_matrix}
            \begin{bmatrix}
                \underline{0}\\
                \partial_{y_{[-R_1]}} g(F_x,F_u)
            \end{bmatrix}=\left(\partial_{x, u}
            \begin{bmatrix}
                f\\
                g
            \end{bmatrix} \circ F_{x,u}\right) \cdot 
            \begin{bmatrix}
                \partial_{y_{[-R_1]}} F_{x}\\
                \partial_{y_{[-R_1]}} F_{u}
            \end{bmatrix}\,.
        \end{equation}
        Since the map \eqref{eq:sysEq_ext} is by construction a diffeomorphism, the corresponding Jacobian matrix in \eqref{eq:rank_matrix} is regular, which implies that the matrices $\partial_{y_{[-R_1]}} g\left(F_{x},F_u\right)$ on the left-hand side and $\begin{bmatrix}
            \partial_{y_{[-R_1]}} F_{x} & \partial_{y_{[-R_1]}} F_{u}
        \end{bmatrix}^T$ on the right-hand side have the same rank.
        Furthermore, after applying an invertible input transformation as in Lemma \ref{lem:Jacobian_transformed}, the corresponding Jacobian matrix is of the form \eqref{eq:Jacobian_transformed}.
        Due to the fact that this specific system structure can always be achieved by a regular input transformation (recall that $\Rank (\partial_u f)=m$) and that input transformations do not affect the parameterization $F_x$ of the state variables, it follows that $\Rank (\partial_{y_{[-R_1]}}g(F_x,F_u))=\Rank (\partial_{y_{[-R_1]}}F_x)$.
    \end{proof}

\end{document}